\documentclass[11pt,a4paper]{amsart}

\usepackage[english]{babel}
\usepackage[T1]{fontenc}
\usepackage[utf8]{inputenc}
\usepackage{lmodern}
\usepackage{amsmath,amsfonts,amssymb}
\usepackage{hyperref}
\usepackage{url}
\usepackage{tikz}
\usetikzlibrary{matrix,arrows}

\author{Gunnar Þór Magnússon}
\address{Institut Fourier\\
100 rue des Maths\\
38402 St. Martin d'Hères\\
France}
\email{gunnar.magnusson@ujf-grenoble.fr}

\newtheoremstyle{gtheorem}
     {3pt}        
     {3pt}        
     {\itshape}           
     {\parindent} 
     {\bfseries}   
     {.}          
     {.5em}       
     {}           

\newtheoremstyle{gdefinition}
     {3pt}        
     {3pt}        
     {}           
     {\parindent} 
     {\bfseries}   
     {.}          
     {.5em}       
     {}           

\newtheoremstyle{gremark}
     {3pt}        
     {3pt}        
     {}           
     {\parindent} 
     {\itshape}   
     {\hspace{0.5em}---}          
     {0.5em}       
     {}           

\theoremstyle{gtheorem}
\newtheorem{theo}{Theorem}[section]
\newtheorem{prop}[theo]{Proposition}

\theoremstyle{gdefinition}

\newtheorem{exam}{Example}

\theoremstyle{gremark}
\newtheorem*{rema}{Remark}

\renewenvironment{proof}{{\itshape Proof: }}{\hfill\qedsymbol\vspace{3pt}}

\newcommand{\Z}{\mathbb{Z}}

\newcommand{\C}{\mathbb{C}}

\renewcommand{\P}{\mathbb{P}}
\newcommand{\cc}[1]{\mathcal{#1}}

\newcommand{\Aut}{\mathop{\text{Aut}}}
\newcommand{\Ric}{\mathop{\text{Ric}}}
\def\fiber{M}
\def\dV{dV}
\def\Vol{\mathop{\text{Vol}}}

\begin{document}

\begin{abstract}
  If $f$ is an automorphism of a compact simply connected K\"ahler
manifold with trivial canonical bundle that fixes a K\"ahler class, then
the order of $f$ is finite. We apply this well known result to construct
compact non K\"ahler manifolds. These manifolds contradict the abundance
and Iitaka conjectures for complex manifolds.
\end{abstract}

\title[Examples of non-K\"ahler manifolds]{Automorphisms and examples of\\ compact non-K\"ahler manifolds}
\maketitle

\section*{Introduction}

Let $X$ be a compact complex manifold of dimension $n$. The generalized
version of the abundance conjecture says that if $X$ is K\"ahler then
the numerical dimension of the canonical bundle $K_X$ should be equal
to its Kodaira dimension \cite[Chapter~18]{analmeth}. 
A consequence of this conjecture is the Iitaka $C_{n,m}$ conjecture, 
which says that if $f : X \to Y$ is a holomorphic morphism of
compact K\"ahler manifolds, then $\kappa(X) \geq \kappa(Y) + \kappa(f_y)$,
where $f_y$ is a general fiber of $f$ and $\kappa$ denotes the Kodaira
dimension.

These conjectures were originally stated for projective varieties,
but their statements make sense for K\"ahler manifolds and indeed
any compact complex manifold. In this paper we produce a examples of 
compact non K\"ahler manifolds that violate both the abundance and the
Iitaka conjectures. That these conjectures fail for non-K\"ahler manifolds has been known for some time; \cite[Remark~15.3]{MR0506253} contains an example of a torus bundle over a torus that contradicts the Iitaka conjecture (it makes no mention of abundance, simply beceause it hadn't been conjectured at the time). The construction of our manifolds is in the same spirit as this previous example.

The construction is simple. A folklore result says that
if $\fiber$ is a simply connected K\"ahler manifold with trivial canonical
bundle that admits an automorphism $f$ of infinite order, then $f$ must
move every K\"ahler class on $\fiber$. Given such a manifold, we let a 
lattice in a complex vector space $V$ act on $\fiber \times V$ by translation
on $V$ and by mapping each generator of the lattice to $f$. The quotient
manifold is then a compact non K\"ahler manifold, with flat canonical bundle,
but whose Kodaira dimension is negative in some cases.

We start by detailing this construction and proving our claims on the
canonical bundle, then we point to beautiful work of Oguiso
that shows that the required K\"ahler manifolds and automorphisms exist.

\smallskip \noindent
\textbf{Acknowledgements.} A thousand thanks to Keiji Oguiso for his
interest in this work, to Valentino Tosatti for questions and remarks 
on an earlier version of this paper, and to Jean-Pierre Demailly for everything.

\section{Automorphisms and K\"ahler classes}

Let $\fiber$ be a compact simply connected K\"ahler manifold of complex 
dimension $\dim_\C \fiber = n$ with trivial canonical bundle. 
Examples of such manifolds include K3 surfaces, Calabi--Yau manifolds 
and hyperk\"ahler manifolds; see \cite{MR730926}.

Let $\omega$ be a K\"ahler metric on $\fiber$. The Ricci curvature of $\fiber$
may be defined as the curvature form of the metric that $\omega$
induces on the canonical bundle of $\fiber$. In local coordinates, one has
$\smash{2\pi \Ric \omega = - i \partial \bar \partial \log \det
  \omega_{j\overline k}}$.
Yau proved in \cite{MR480350} that if $[\omega]$
is a K\"ahler class on $\fiber$, then there exists a unique Ricci-flat
K\"ahler metric $\omega$ in the class $[\omega]$. The existence of such
metrics has great consequences for the geometry of the manifold
$\fiber$, for example:

\begin{prop}
  An automorphism $f$ of $\fiber$ fixes a K\"ahler class $[\omega]$ on $\fiber$ if and only if the order of $f$ is finite.
\end{prop}

\begin{proof}
  The condition is clearly sufficient, since if the degree of $f$
is $d$ then the K\"ahler class $\smash{[\omega] + f^*[\omega] + \dots +
(f^*)^{d-1} [\omega]}$ is invariant under $f$.

Suppose now that $f$ fixes a K\"ahler class $[\omega]$ and let
$\omega$ be the unique Ricci flat metric in this class. Then
$f^*\omega$ is again Ricci flat, and thus equal to $\omega$
by unicity. Thus $f$ is an element of the isometry group of
$(\fiber,\omega)$. A general result of Riemannian geometry 
\cite[Corollary~6.2]{MR2243012} now says that the isometry group of a 
simply connected manifold with non positive Ricci curvature is finite.
\end{proof}

  The condition that $\fiber$ be simply connected serves to exclude complex
tori, for tori admit nonzero holomorphic vector fields. 
These fields generate automorphisms homotopic to the identity,
which thus act trivially on the cohomology of the torus, despite
usually being of infinite order.

This result points the way to a construction of non-K\"ahler
manifolds: Let $\fiber$ be a compact simply connected K\"ahler manifold with
trivial canonical bundle. Suppose $\fiber$ admits
an automorphism $f$ of infinite order. Let $V$ be a complex vector
space of dimension $p$ and let $\Gamma$ be a lattice in $V$, we
denote by $B = V / \Gamma$ the complex torus defined by $\Gamma$.
We define a representation $\Gamma \longrightarrow \Aut \fiber$
by mapping every generator of $\Gamma$ to the automorphism $f$. The
lattice $\Gamma$ then acts on the product $\fiber \times V$ by
$$\gamma \cdot (z,t) = (\gamma(z), t + \gamma).$$
We set $X := X(\fiber,B) =  (\fiber \times V) / \Gamma$.

\begin{prop}
  The complex space $X$ is a smooth compact non K\"ahler manifold. It
  is the total space of a holomorphic fibration $\pi : X \to B$, 
whose fibers are all isomorphic to $\fiber$.
\end{prop}

\begin{proof}
  The lattice $\Gamma$ clearly acts without fixed points on 
$\fiber \times V$. Its action is also properly discontinuous, since any compact
set in $\fiber \times V$ may be translated as far to infinity in $V$
as desired. The quotient $X$ is thus a smooth complex manifold, and
compact for the same reason that the torus $V / \Gamma$ is compact.

The projection map $pr : \fiber \times V \longrightarrow V$ is invariant
by the action of $\Gamma$ and thus defines a holomorphic morphism
$\pi : X \to B$. It is proper as the manifold $X$ is
compact, and a submersion because the projection morphism is a
submersion. Let $t$ be a point of $B$. The preimage $\pi^{-1}(t)$
may be identified with the product $\fiber \times \Gamma + t$. If we pick
an element $\gamma$ in the lattice $\Gamma$, then the restriction
of the quotient map $q : \fiber \times V \to X$ identifies with the 
automorphism $\gamma \cdot f : \fiber \to \fiber$ and defines an 
isomorphism $\fiber \to X_t$.

Finally, suppose that $X$ were K\"ahler. If $\omega$ were a K\"ahler
metric on $X$, then by restriction we would obtain a K\"ahler class
$[\omega_0]$ on the fiber $\fiber_0$ that would be invariant under the action of the monodromy on the cohomology of $\fiber_0$. But the monodromy group is the same as the group generated by $f$, so this is impossible since $f$
is of infinite order.
\end{proof}

\begin{rema}
  It seems hard to extract precise topological information
about $X$, aside from that which follows trivially from general facts about 
fibrations. For example, the naive road to the Betti numbers of $X$
passes through the space of closed forms on $\fiber \times V$ that
are invariant under the automorphism $f$. Since $f$ is quite wild
I have no idea how one could calculate this in practice.
\end{rema}

The canonical bundle of $\fiber$ is trivial, so there is a 
nowhere zero holomorphic $(n,0)$-form $\sigma$ on $\fiber$. 
As $f^*\sigma$ is again
a $(n,0)$-form on $\fiber$, we must have $f^* \sigma = \lambda \sigma$ for some
complex number $\lambda$. Note that the $(n,n)$-form $i^{n^2} \sigma \wedge
\overline \sigma$ is real and positive on $\fiber$, and that $f^*(\sigma \wedge
\overline \sigma) = |\lambda|^2 \sigma \wedge \overline \sigma$.
Integrating over $\fiber$, we find $|\lambda| = 1$.

\begin{prop}
  The Kodaira dimension of $X$ is zero if $\lambda$ is a root of 
unity and negative otherwise.
\end{prop}

\begin{proof}
  Suppose $\alpha$ is a global section of $m K_X$ for some $m \geq 1$.
If $q : \fiber \times V \longrightarrow X$ is the quotient map, then $q^*\alpha$
is a global section of $m K_{\fiber \times V}$. We may thus write
\begin{align*}
  q^*\alpha = \theta(z,v) \,
  \bigl(
  \sigma_\fiber \otimes \sigma_V
  \bigr)^{\otimes m},
\end{align*}
where $\sigma_\fiber$ is a trivializing section of $K_\fiber$, 
$\sigma_V = d v_1 \wedge \ldots \wedge d v_n$ is the standard holomorphic 
volume form on $V$, and $\theta$ is a holomorphic function on $\fiber \times V$.
We note that since $\fiber$ is compact, $\theta$ is actually just
a holomorphic function on $V$.

Since $\alpha$ is a section of $m K_X$, the pullback $q^*\alpha$ must be
invariant under the action of $\Gamma$ on $\fiber \times V$.
The holomorphic volume form $\sigma_V$ is invariant under the action 
of $\Gamma$,
so if $\gamma_i$ is one of the generators of $\Gamma$ we find 
\begin{align*}
  \theta(v) \,
  \bigl(
  \sigma_\fiber \otimes \sigma_V
  \bigr)^{\otimes m}
  = q^*\alpha
  = \gamma_i \cdot q^*\alpha
  = \lambda^m \theta(v + \gamma_i)
  \bigl(
  \sigma_\fiber \otimes \sigma_V
  \bigr)^{\otimes m}.
\end{align*}
If $\gamma = \sum_i a_i \gamma_i$ is an element of $\Gamma$, we set
$\mathop{\text{deg}}\gamma := \sum_i a_i$. Using the above we then get
$\theta(v) = \lambda^{m\mathop{\text{deg}}\gamma} \theta(v + \gamma)$ for
any $\gamma$ and $v$. This entails that $|\theta(v)| = |\theta(v + \gamma)|$
for all $v$ and $\gamma$, but then $|\theta|$ takes its maximum on $V$
in the fundamental paralleogram of $\Gamma$, so $\theta$ is constant.
The complex number $\lambda$ must then satisfy $\lambda^m = 1$.

We thus see that if $\lambda$ is an $m^{th}$ root of unity, then every
$m^{th}$ power of $K_F$ admits a unique non-zero holomorphic section, so
 the Kodaira dimension of $X$ is zero. Likewise, if $\lambda$
is not a root of unity, then no power of $K_\fiber$ admits a global section,
so the Kodaira dimension of $X$ is negative.
\end{proof}

\begin{prop}
  The numerical dimension of $K_X$ is zero.
\end{prop}

\begin{proof}
  We will show that the canonical bundle $K_X$ admits a flat hermitian
metric. Its first Chern class is thus zero, which implies the proposition.

  Since $\fiber \to X \to B$ is a fibration there is a short exact sequence
  \begin{align*}
    0 \longrightarrow T_{X/\fiber} 
    \longrightarrow T_X
    \longrightarrow \pi^* T_B
    \longrightarrow 0
  \end{align*}
of tangent bundles over $X$. Note that since $B$ is a torus the bundle
$\pi^*T_B$ is trivial. The adjunction formula now says that the canonical
bundle of $X$ is $K_X = K_{X/\fiber}$. Let $q : \fiber \times V \to X$ be 
the quotient morphism and consider the pullback bundle 
$q^*K_{X/\fiber} = p_\fiber^*K_\fiber$, where 
$p_\fiber : \fiber \times V \to \fiber$ is the projection.

Now pick a Ricci-flat K\"ahler metric $\omega$ on $\fiber$, and let
$\dV = \omega^n/n!$ be its volume form. Recall that the volume form
of any other Ricci-flat K\"ahler metric is a constant multiple of $dV$.
The form $dV$ defines a smooth hermitian metric on $p^*_\fiber K_\fiber$ 
by the formula 
$h(\alpha, \overline \beta) \, dV = i^{n^2} \alpha \wedge \overline \beta$,
where $\alpha$ and $\beta$ are sections of $p^*_\fiber K_\fiber$.
The curvature form of this metric is the Ricci-form of $\omega$, so it
is flat.

If $\sigma_\fiber$ is a trivializing holomorphic volume form on $\fiber$,
then $f^*\sigma_\fiber = \lambda \sigma_\fiber$, where $\lambda$
is a complex number with absolute value 1. Also note that $f^*\omega$ is
again a Ricci-flat K\"ahler metric on $\fiber$, and that 
\begin{align*}
  \Vol(\fiber,f^*\omega) = \int_\fiber \frac{f^*\omega^n}{n!}
  = \int_\fiber \frac{\omega^n}{n!} = \Vol(\fiber,\omega)
\end{align*}
because $f : \fiber \to \fiber$ is a surjective finite morphism of degree one. 
Thus $f^*\dV = \dV$. From these two facts it follows that
\begin{align*}
  f^*
  \bigl(
  h(\alpha, \overline{\beta})
  \bigr) \dV
  = f^*
  \bigl(
  h(\alpha, \overline{\beta}) \dV
  \bigr)
  = i^{n^2} f^*\alpha \wedge \overline{f^*\beta}
  = h(f^*\alpha, \overline{f^*\beta}) \dV,
\end{align*}
so the metric $h$ is invariant under the action of $\Gamma$ and thus defines
a flat hermitian metric on $K_{X/\fiber} = K_X$.
\end{proof}

\section{Automorphisms of hyperk\"ahler manifolds}

As before we let $\fiber$ be a compact simply connected K\"ahler manifold 
with trivial canonical bundle.
The automorphism group of $\fiber$ admits a natural representation
$$\Aut \fiber \longrightarrow \Aut H^2(\fiber,\C),$$
obtained by sending each automorphism to the pullback morphism on
cohomology. If $\fiber$ is a K3 surface, then the global Torelli theorem
entails that this group morphism is actually injective. The order
of an automorphism $f$ is thus equal to the order of its pullback
$f^*$ on degree two cohomology.

One may obtain examples of higher dimensional holomorphic symplectic
manifolds from a K3 surface, see \cite{MR730926}. The idea is to
consider the symmetric product $\fiber^n / \mathfrak S_n$. This space
is singular, but the Douady space $\smash{\fiber^{[n]}}$ of subspaces of $\fiber$
of length $n$ is a desingularization of the symmetric product. The
Douady space is then a holomorphic symplectic manifold of dimension
$2n$.

The second cohomology of the Douady space is isomorphic to
\begin{align*}
  H^2(\fiber^{[n]},\C) =
  H^2(\fiber,\C) \oplus \C \cdot E,
\end{align*}
where $E$ is an exceptional divisor of the desingularization $\smash{\fiber^{[n]}
\to \fiber^n / \mathfrak S_n}$. Any automorphism $f$ of the K3 surface $\fiber$
induces an automorphism of the Douady space $\smash{\fiber^{[n]}}$. This new
automorphism acts like $f$ on the part of the second cohomology
coming from $\fiber$, and trivially on the exceptional divisor. In
particular, if $f$ is of infinite order on $\fiber$, then the induced
automorphism on $\smash{\fiber^{[n]}}$ is of infinite order.

Recall that the holomorphic symplectic form $\sigma$ on $\fiber$ is unique
up to scalars. It follows that $\sigma$ is an eigenvector of any
automorphism $f$ of $\fiber$, and as before one sees that the
eigenvalue of $\sigma$ must have absolute value 1. Oguiso gives much
more precise results in \cite{MR2406267}; for the moment we will
contend ourselves with the following special case of his Theorem 2.4:

\begin{prop}
  Let $\fiber$ is a projective K3 surface and $f$ an automorphism of
$\fiber$. Let $\lambda$ be the eigenvalue of $f^*$ on the space 
$H^0(\fiber,K_\fiber)$. Then $\lambda$ is a root of unity.
\end{prop}

By our discussion of Douady spaces, the same is true of the
holomorphic symplectic space constructed from a projective K3
surface.

\begin{exam}
  Let $\P := \P^1_1 \times \P^1_2 \times \P^1_3$. This space comes
  equipped with three projections $p_i : \P \to \P_i$. Let
$$
L := 
p_1^* \, \cc O(1)_{\P^1_1} \otimes
p_2^* \, \cc O(1)_{\P^1_2} \otimes
  p_3^* \, \cc O(1)_{\P^1_3}
$$
be an ample line bundle on $\P$, so that $K_\P = -2L$. The adjunction
formula shows that if $\tau$ is a general section of $2L$, then
the zero variety $X = \tau^{-1}(0)$ is a smooth K3 surface.

We can now consider the projections $p_{jk} : \P \to \P_j \times
\P_k$. Restricted to the K3 surface $\fiber$, these define ramified
coverings $\fiber \to \P_j \times \P_k$ of degree 2. The Galois groups
of these coverings give three holomorphic involution $\iota_i$
of $\fiber$, and we have
\begin{align*}
  \Aut X =
  \langle \iota_1, \iota_2, \iota_3 \rangle
  \simeq \Z_2 * \Z_2 * \Z_2,
\end{align*}
where $\Z_2 := \Z / 2\Z$.
Both identities in the above formula are nontrivial, but they are
proved in \cite{Oguiso:birat}.
The automorphism group of $\fiber$ thus contains several elements of
infinite order.
\end{exam}

\begin{exam}
  We again refer to Oguiso's paper \cite[Examples~2.5 and~2.6]{MR2406267}, 
from which
one may extract that there exists a K3 surface $\fiber$ which admits
an automorphism $f$ such that the eigenvalue of $f^*$ on $H^0(\fiber,K_\fiber)$
has infinite order. As before, it follows that there exist higher 
dimensional hyperk\"ahler manifolds with the same property.

We now consider the non K\"ahler manifold $X = X(\fiber,B)$. This
manifold has negative Kodaira dimension by our earlier results.
By construction there is a holomorphic map $\pi : X \longrightarrow B$
whose fiber at every point is $\fiber$. Both $\fiber$ and $B$ have Kodaira
dimension zero, so
\begin{align*}
  \kappa(X) < \kappa(\fiber) + \kappa(B).
\end{align*}
The manifold $X$ is thus shows that the Iitaka $C_{n,m}$ conjecture is false
for general compact complex manifolds. Since $\kappa(X)$ is
negative but the canonical bundle $K_X$ has numerical dimension zero,
the manifold also shows that the generalized abundance conjecture
is false for general complex manifolds.
\end{exam}

\begin{exam}
  Oguiso and Schr\"oer show in \cite{Oguiso:enriques} that the
  universal cover of the Douady space $\fiber^{[n]}$ of an Enriques
  surface $\fiber$ is a Calabi--Yau manifold. They also show that there
  exists an Enriques surface $\fiber$ with $\Aut \fiber = \Z_2 * \Z_2 * \Z_2$,
  similarly to the hyperk\"ahler manifolds considered above. The
  fibration $X(\fiber^{[n]},B)$ then provides an example of a non-K\"ahler
  manifold $X \to B$ with a Calabi--Yau fiber.
\end{exam}

\bibliographystyle{alpha}
\bibliography{abundance}

\end{document}